\documentclass[a4paper]{article}

\usepackage{amsmath,amsthm}

\usepackage{newtxtext}
\usepackage[subscriptcorrection]{newtxmath}
\usepackage{color}

\usepackage{natbib}
\usepackage{geometry}

\usepackage[plain,noend]{algorithm2e}

\makeatletter
\renewcommand{\algocf@captiontext}[2]{#1\algocf@typo. \AlCapFnt{}#2} 
\def\@algocf@capt@plain{top}
\renewcommand{\algocf@makecaption}[2]{
  \addtolength{\hsize}{\algomargin}%
  \sbox\@tempboxa{\algocf@captiontext{#1}{#2}}%
  \ifdim\wd\@tempboxa >\hsize
    \hskip .5\algomargin%
    \parbox[t]{\hsize}{\algocf@captiontext{#1}{#2}}
  \else%
    \global\@minipagefalse%
    \hbox to\hsize{\box\@tempboxa}
  \fi%
  \addtolength{\hsize}{-\algomargin}%
}
\makeatother

\newtheorem{theorem}{Theorem}

\newtheorem{proposition}{Proposition}
\newtheorem{definition}{Definition}
\newtheorem{remark}{Remark}
\newtheorem{algo}{Algorithm}

\def\prob{{\rm pr}}
\def\support{{\rm supp}}
\def\vspan{{\rm span}}
\def\adj{\dag}
\def\closedBall{\overline{B}}
\def\sphere{\mathbb{S}}
\def\hsphere{{\mathcal{S}}}

\def\expectation{{E}}
\def\Real{{\mathbb{R}}}
\def\H{{\mathcal{H}}}
\def\ud{{\mathrm{d}}}
\newcommand\RpNorm[1]{\lvert #1 \rvert}
\newcommand\norm[1]{\lVert #1 \rVert}
\newcommand\inner[1]{\langle #1 \rangle}

\def\tp{{ \mathrm{\scriptscriptstyle T} }}

\begin{document}

%


\title{On the existence and non-existence of centres of mass on Hilbert spheres}

\author{Shahin Tavakoli \\ RISIS, GSEM, University of Geneva\\ \texttt{shahin.tavakoli@unige.ch}}

\maketitle

\begin{abstract}
Fréchet means and $L^p$ centres of mass provide notions of average location in metric spaces. On finite-dimensional spheres, existence follows from compactness.
 On infinite-dimensional spheres, it is not known whether a centre of mass always exists. We show that this is not always the case, and give a simple assumption under which a centre of mass exists. We then show that finding the sample centre of mass of data $x_1, \ldots, x_n$ on the sphere is always an optimisation problem on a subsphere of manifold dimension at most $n$, regardless of the potentially infinite dimension of the sphere. We conclude with some statistical implications.
\end{abstract}


\section{Introduction}

\subsection{Centres of mass}

The notion of `location' for a random quantity $X$ is a fundamental concept in statistics. For a random variable or random vector, this is typically given by the expectation. Taking an expectation assumes that $X$ lies in a vector space. However many interesting statistical problems consider data that lie in spaces that are non-Euclidean, where taking averages is not a meaningful operation.

Non-Euclidean data arise in many areas of statistics, such as statistical shape analysis \citep{kendallShapeManifoldsProcrustean1984,leLocatingFrechetMeans2001,drydenGeometryDrivenStatistics2015}, geometry-driven statistics \citep{drydenGeometryDrivenStatistics2015,pennecRiemannianGeometricStatistics2020a}, and elastic shape analysis \citep{srivastavaFunctionalShapeData2016}. These can be viewed as data on Riemannian manifolds or more generally data on metric spaces. 

In such non-Euclidean situations, a notion of location can be generalised by noticing that the expectation $\mu = \expectation(X)$ of a random vector $X \in \Real^k$ is the minimiser of the functional $F(m) = \expectation \RpNorm{m - X}^2 , m \in \Real^k$, where $\RpNorm{x} = (x^\tp x)^{1/2}$.
This leads to the notion of Fr\'echet mean on metric spaces:
\begin{definition}[\citet{Frechet1948}]
    Let $(M, d)$ be a metric space, and $X \in M$ a random element. Let $F_2(m) = \expectation( d^2(m,X) )$, $m \in M$.
    Assuming $\expectation( d^2(X,m) ) < \infty$ for some $m \in M$, the Fr\'echet mean of $X$ is
    \[
        \mu_2(X) := \arg\inf_{m \in M} F_2(m).
    \]
    More generally, the $L^p$ centres of mass \citep{afsariRiemannianCenter2011} are given by 
    \[
        \mu_p(X) := \arg\inf_{m \in M} F_p(m), \quad p \geq 1,
      \]
      where $F_p(m) = \expectation( d^p(m,X) )$. 

      The $\phi$-centres of mass, for $\phi:[0, \infty) \to \Real$, are obtained by replacing $F_p$ by $F_\phi(m) = \expectation( \phi( d(m,X) ) )$ in the last displayed equation. 
\end{definition}
In full generality $\mu_p(X)$ could be empty, a singleton, or a set with multiple elements. The class of $\phi$-centres of mass includes the Huber mean \citep{leeHuberMeansRiemannian2025} as a special case.

Fr\'echet means were introduced by \citet{Frechet1948}. \citet{karcherRiemannianCenterMass1977} considered the local minimisers of $F_2$, known as the Karcher mean, which was further studied in \citet{kendallProbabilityConvexityHarmonic1990}. \citet{afsariRiemannianCenter2011} gives general sufficient conditions for uniqueness of $L^p$ centres of mass in finite-dimensional Riemannian manifolds.

Among the different metric spaces of interest, the sphere (of finite or infinite dimensions) is of particular interest: it is (arguably) the simplest nonlinear manifold, and arises in several applications: directional data are data on spheres, pre-shape spaces are spheres, densities can be viewed as points on an infinite-dimensional sphere, and in elastic shape analysis, the pre-shape space of square-root velocity functions is an infinite-dimensional sphere \citep{srivastavaFunctionalShapeData2016}. See also \citet{drydenStatisticalAnalysisHighdimensional2005} for another approach to infinite-dimensional spherical data as a limit of finite-dimensional spherical data.

\subsection{Finite-dimensional spheres}

In the finite-dimensional sphere $\sphere^k := \{ x \in \Real^{k+1} \mid \RpNorm{x}^2 = 1 \}$, equipped with the distance $d(x,q) = \arccos(x^\tp q)$ for $x,q \in \sphere^k$, the $L^p$ centre of mass set is always non-empty, by compactness of $\sphere^k$ and continuity of the functional $F_p$. If $X \sim \nu$ has support inside a spherical cap of radius strictly smaller than $\pi/2$, i.e., $\support(\nu) \subseteq \closedBall(o, r) := \{ x \in \sphere^k \mid d(o,x) \leq r\}$ with $r < \pi/2$, the Fr\'echet mean is unique and belongs to $\closedBall(o,r)$ \citep{kendallProbabilityConvexityHarmonic1990,kendallConvexityHemisphere1991,afsariRiemannianCenter2011}. Without further distributional assumptions, the radius condition $r < \pi/2$  is sharp: for $r = \pi/2$ the distribution with mass $1/2$ at two antipodal points has an infinite number of Fréchet means.

\subsection{Infinite-dimensional spheres}

Matters become more complicated in the infinite-dimensional sphere $\hsphere := \{ x \in \H \mid \norm{x}^2=1 \}$, where $\H$ is an infinite-dimensional separable Hilbert space with inner product $\inner{\cdot, \cdot}$ and norm $\norm{\cdot} = \inner{\cdot, \cdot}^{1/2}$, the canonical example of which is the space $\H = L^2([0,1], \Real)$ of square-integrable functions $f:[0,1]\to \Real$, $\norm{f}^2 = \int f^2(x) \ud x < \infty$.

In an infinite-dimensional sphere $\hsphere$, even existence of a Fr\'echet mean is non-trivial, because $\hsphere$ is not compact. If the support of $X \in \hsphere$ lies within a spherical cap $\closedBall(o, r) := \{ x \in \hsphere \mid d_\hsphere(o,x) \leq r \}$ of radius $r < \pi/2$, where $d_\hsphere(x,x') = \arccos( \inner{x, x'} )$ is the great-arc distance on $\hsphere$, then \citet[Theorem~B and Theorem~57]{yokotaConvexFunctionsBarycenter2017} showed that the $L^p$ centre of mass ($p \geq 2$) exists, is unique, and belongs to the convex hull of the support of $X$ \citep[the proofs use techniques from][]{kendallConvexityHemisphere1991}. In particular, this implies that finding the sample Fr\'echet mean of $x_1,\ldots, x_n \in \closedBall(o,r) \subset \hsphere$ where $o \in \hsphere, r < \pi/2$, i.e., the minimisation of $F_p^{(n)}(m) = n^{-1} \sum_{i=1}^n d_\hsphere^p(m, x_i)$, is an optimisation over a manifold of dimension at most $n-1$.
The paper \citet{yokotaConvexFunctionsBarycenter2017} 
appears to have received limited attention in the statistical literature,
perhaps because of its very general setting (it considers random elements in CAT$(1)$ spaces). Another paper worth mentioning is \citet{daiStatisticalInferenceHilbert2022}, which shows existence and uniqueness of the Fr\'echet mean in the infinite-dimensional sphere if the support of the distribution has diameter at most $\pi/2$, and shows that this implies that the distribution has support within $B(\mu, \pi/2)$, the open ball of radius $\pi/2$ centred at the unique Fr\'echet mean $\mu$. 

Without the spherical cap assumption---i.e. without $\prob(X \in \closedBall(o,r) ) = 1$ for some $o \in \hsphere$ and $r < \pi/2$---it is not yet known whether a centre of mass on the infinite-dimensional sphere exists or not, and it is not known whether the problem of computing the sample $L^p$ centre of mass is an infinite-dimensional optimisation problem or reduces to a finite-dimensional one. In this paper, we show that a centre of mass does not always exist on the infinite-dimensional sphere $\hsphere$. We then provide a simple condition under which a centre of mass exists. We show that finding a sample centre of mass is always a finite-dimensional optimisation problem, and describe how its computation can be implemented as a finite-dimensional optimisation problem. All proofs are given in the appendix.

\section{Main results}

Our first result is that the $L^p$ centre of mass does not always exist on the infinite-dimensional sphere $\hsphere \subset \H$.

\begin{proposition}
    \label{prop:non-existence}
  For any $p>1$, there exists a distribution with empty $L^p$ centre of mass. 
\end{proposition}

A concrete distribution on $\hsphere$ not admitting an $L^p$ centre of mass is the following. Take $\{ u_n : n = 1,2, \ldots \} \subset \H$ a complete orthonormal sequence, and weights $w_n \in (0,1/2), n =1,2, \ldots$ summing to $1/2$. The probability distribution $\nu$ on $\hsphere$ given by $\nu = \sum_{n=1}^\infty w_n(\delta_{-u_n} + \delta_{u_n})$, where $\delta_u$ is a Dirac mass at $u$, does not admit an $L^p$ centre of mass for any $p > 1$ (see the proof of Proposition~\ref{prop:non-existence}).
One can also construct counter-examples that are not discrete, by replacing $(\delta_{-u_n} + \delta_{u_n})$ with uniform distributions on finite-dimensional and orthogonal sub-spheres of $\hsphere$.

The idea underlying this example is that the centre of mass wants to `escape' each subspace $\vspan(u_n)$. This pushes the centre of mass to be orthogonal to each $u_n$, which is impossible. The next result shows that if we allow for a `dimension to escape to', a centre of mass always exists.
\begin{theorem}
\label{thm:existence}
Let $\nu$ be a probability measure on $\hsphere$ 
and let $U = \overline \vspan\{ \support(\nu) \}$. 

If $\dim\{ U^\perp \} \geq 1$, then for any continuous function $\phi:[0,\pi] \to \Real$, there exists a $s^* \in \hsphere$ such that
\[
    F_\phi(s^*) = \inf_{s \in \hsphere} F_\phi(s),
\]
where $F_\phi(s) = \int_\hsphere \phi\{d_\hsphere(s,x)\} \ud \nu(x)$.
Moreover, if such an $s^*$ is unique, then $s^* \in U \cap \hsphere$.
In particular, an $L^p$ centre of mass of $\nu$ exists for any $p > 0$. 

Additionally, the result holds for $L^\infty$ centres of mass, that is, minimisers of 
\[
  F_\infty(s) = \sup_{ x \in \overline{\support(\nu)} } d_\hsphere(s,x), \quad s \in \hsphere.
\]
\end{theorem}
A simple application of Theorem~\ref{thm:existence} is to the sample $L^p$ centre of mass of $x_1,\ldots, x_n \in \hsphere$, i.e. the minimiser of $F_p^{(n)}(s) = n^{-1} \sum_{i=1}^n d^p_\hsphere(s, x_i)$. If the minimiser is unique, it belongs to $\vspan(x_1,\ldots, x_n) \cap \hsphere$, and hence finding the $L^p$ centres of mass is an optimisation over a sphere of manifold dimension at most $n-1$.  This result extends \citet[][Theorem 57]{yokotaConvexFunctionsBarycenter2017} to samples not lying within a spherical cap $\closedBall(o, r)$ with radius $r<\pi/2$, but assumes the centre of mass is unique. The following result is an extension to non-unique centres of mass.
\begin{theorem}
    \label{thm:sample}
    Let $x_1,\ldots, x_n \in \hsphere$, and pick any $u \in \vspan(x_1,\ldots, x_n)^\perp \cap \hsphere$.
For any continuous function $\phi:[0,\pi] \to \Real$, a minimiser $\hat s$ of 
\[
  F_\phi^{(n)}(s) =  n^{-1} \sum_{i=1}^n \phi\{d_\hsphere(s,x_i)\}, \quad s \in \hsphere,
\]
exists and can be found within the finite-dimensional subsphere $\vspan(u, x_1,\ldots, x_n) \cap \hsphere$. 
In particular, finding a sample $L^p$ centre of mass for $p>0$ is an optimisation over a manifold of dimension at most $n$.
Furthermore, if $\inner{\hat s, u} \neq 0$ for a minimiser $\hat s \in \hsphere$ of $F^{(n)}_\phi$, then $F^{(n)}_\phi$ has multiple minimisers.  
\end{theorem}
The results of Theorem~\ref{thm:sample} also hold for the finite-dimensional sphere $\sphere^k$, $k \geq 1$.

\section{Statistical Implications}

Any implementation of the search for sample Fréchet means for infinite-dimensional data $x_1,\ldots, x_n \in \hsphere$ requires some form of discretisation. If $\hsphere$ is the sphere in $L^2([0,1], \Real)$ then some popular approaches are to evaluate each data point on a fine grid of values $t_1,\ldots, t_m \in [0,1]$ to obtain finite-dimensional representations $c_i = (x_i(t_1), \ldots, x_i(t_m) ) \in \Real^m$, or to use finite basis expansions to represent $x_i$ by a vector $c_i = (c_{i1},\ldots, c_{im}) \in \Real^m$ for which $x_i(t) \approx c_{i1} \psi_1(t) + \ldots + c_{im} \psi_m(t), \; t \in [0,1]$. If the data are complex (e.g., the $x_i$ are densities that exhibit rapid changes), then large values of $m$ could be needed. Theorem~\ref{thm:sample} implies that the search for sample Fréchet means is an optimisation over a subsphere of manifold dimension at most $n$, which can be advantageous if $m \gg n$. Algorithm~\ref{algo:computation_centre_mass} gives the details to achieve this.

\begin{algo} \label{algo:computation_centre_mass}
Take an arbitrary $u \in \hsphere \cap \vspan(x_1,\ldots, x_n)^\perp$.
  \begin{tabbing}
    \qquad Compute Gram matrix $G \in \Real^{n \times n}$ with $(G)_{ij} = x_i^\tp x_j$\\
    \qquad  Compute eigendecomposition $G = U \Lambda U^\tp$\\
    \qquad  Truncate $U \in \Real^{n \times r}$ and $\Lambda \in \Real^{r\times r}$ where $r = $ number of non-zero eigenvalues of $G$ \\
    \qquad For $i=1,\ldots, n$ set $y_i \in \Real^{r+1} \leftarrow$ $i$-th row of $U \Lambda^{1/2}$ appended with $0$ \\
    \qquad  Apply a suitable optimisation method \citep[e.g.,][]{afsariGradient} on $\{y_i\}_{i=1}^n$.\\
    \qquad \qquad Start algorithm from $v = (0,\ldots, 0, 1)^\tp \in \Real^{r+1} $\\
    \qquad \qquad Retain a global minimiser $\tilde m$, the sample $\phi$-centre of mass of $\{y_i\}_{i=1}^n$\\
    \qquad Find any $\tilde \alpha_1,\ldots, \tilde \alpha_n,\tilde  \beta \in \Real$ such that $\tilde m = \sum_{j=1}^n \tilde \alpha_j y_j + \tilde \beta v$ (e.g., using ordinary least squares) \\
    \qquad  \enspace Output: $\phi$-centre of mass $\hat m = \sum_{j=1}^n \tilde \alpha_j x_j + \tilde \beta u$  of $\{x_i\}_{i=1}^n$ \\
  \end{tabbing}
\end{algo}

In Algorithm~\ref{algo:computation_centre_mass}, the iterative algorithm for computing the centre of mass of the reduced data $\{y_i\}_{i=1}^n$ should be started at the vector $v = (0,\ldots, 0, 1)^\tp \in \Real^{r+1}$ to make sure that the algorithm does not  get stuck in $\vspan(y_1,\ldots,y_n)$.

The following result justifies Algorithm~\ref{algo:computation_centre_mass}. 
\begin{proposition}
\label{proposition:gram_to_solution}
The $\hat m = \sum_{j=1}^n \tilde \alpha_j x_j + \tilde \beta u$ obtained in Algorithm~\ref{algo:computation_centre_mass} is the $\phi$-centre of mass of $\{x_i\}_{i=1}^n$.
\end{proposition}
\begin{remark}
  A simple indicator of non-uniqueness of the sample $\phi$-centre of mass is to consider the $\beta$ in Algorithm~\ref{algo:computation_centre_mass}. If $\beta \neq 0$, Theorem~\ref{thm:sample} implies that the sample $\phi$-centre of mass is not unique.
\end{remark}

\section{Discussion}

The questions of strong consistency and asymptotic distributions of the sample centres of mass on the infinite-dimensional sphere without spherical cap assumptions have not been considered here. \citet{evansLimitTheoremsFrechet2024} shows weak convergence of the centres of mass in separable metric spaces, and provides Kuratowski upper limit inclusion of the sample centres of mass sets \citep[see also][for the earliest such result]{ziezoldExpectedFiguresStrong1977}. \citet{daiStatisticalInferenceHilbert2022} provides a strong consistency result and asymptotic distribution for the sample Fréchet mean under bounded diameter assumptions. \citet{jaffeFrechetMeansInfinite2026} considers asymptotic theory for centres of mass in infinite-dimensional metric spaces, but its setup does not apply to the infinite-dimensional sphere.

\section*{Acknowledgement}

The author thanks Almond Stoecker, Victor-Emmanuel Brunel, Adam Jaffe, Yoav Zemel, Victor Panaretos and Stephan Huckemann for helpful discussions. 

\appendix

\section*{Appendix: Proofs of the main results}

  \begin{proof}[of Proposition~\ref{prop:non-existence}]
    Let $p>1$, take $\{ u_n : n = 1,2, \ldots \} \subset \H$ a complete orthonormal sequence, and weights $w_n \in (0,1/2), n =1,2, \ldots$ summing to $1/2$. Define the probability measure $\nu$ on $\hsphere$ by $\nu = \sum_{n=1}^\infty w_n(\delta_{-u_n} + \delta_{u_n})$, where $\delta_u$ is a Dirac mass at $u$. For any $s \in \hsphere$, let $a_n = \inner{s, u_n} \in [-1,1]$ and $\alpha_n = \arccos(a_n) \in [0, \pi]$. Using $\arccos(-t) = \pi - \arccos(t)$, we have 
    \begin{align*}
        F_p(s) &= \int_\hsphere d_\hsphere^p(s,x) \ud \nu(x) 
             \\&= \sum_{n=1}^\infty w_n( \arccos^p( \inner{s, -u_n} ) +  \arccos^p( \inner{s, u_n} ) )
             \\&= \sum_{n=1}^\infty w_n( \arccos^p( -a_n ) +  \arccos^p( a_n ) )
             \\&= \sum_{n=1}^\infty w_n( (\pi - \alpha_n)^p +  \alpha_n^p )
             \\&\geq \sum_{n=1}^\infty 2 w_n (\pi/2)^p
             \\&= (\pi/2)^p,
    \end{align*}
    where we have used the fact that the minimum of $t \mapsto (\pi - t)^p + t^p$ is obtained at $t=\pi/2$ for any $p > 1$. Taking the sequence $s_m = u_m$, we get
    \[
        F_p(s_m) = \sum_{n=1}^\infty w_n( (\pi - \pi/2)^p +  (\pi/2)^p ) - 2w_m (\pi/2)^p + w_m \pi^p  \to (\pi/2)^p, \quad \text{as } m \to \infty
    \]
    Therefore $\inf_{s \in \hsphere} F_p(s) = (\pi/2)^p$, and $F_p(s) = (\pi/2)^p$ if and only if $\inner{s, u_n} = 0$ for all $n$. Since $\{u_n\}$ is complete, this implies  $s = 0$, which does not belong to $\hsphere$.
    This shows that the $L^p$ centres of mass, $p>1$, do not exist in this case.
\end{proof}

\begin{proof}[of Theorem~\ref{thm:existence}]
    Let $U = \overline \vspan\{ \support(\nu) \}$. 
    We extend the domain of $F_\phi$ and $d_\hsphere(\cdot, x)$ to the closed unit ball in $\H$ using the formula $d_\hsphere(u, x) = \arccos \inner{u,x}$ which is valid for any $x \in \hsphere$ and $u \in \H, \norm{u} \leq 1$.
    Any $s \in \hsphere$ can be uniquely decomposed as $s = u + u^\perp$ where $u \in U$, $u^\perp \in U^\perp$. Denote $u = Ps$ where $P: \H \to \H$ is the orthogonal projection onto $U$.
    For $\nu$-almost every $x$, $\inner{s,x} = \inner{Ps,x}$ and thus $\phi\{ d_\hsphere(s,x) \} = \phi\{ d_\hsphere(Ps,x) \}$. 
    Notice that $\norm{Ps} \leq \norm{s}=1$, hence $Ps \in B$, where $B = \{ v \in U \mid \norm{v} \leq 1 \}$ is the closed unit ball in $U$. 
    Furthermore, for any $u \in B$, there exists an $s \in \hsphere$ such that $Ps = u$. Indeed, take $u^\perp \in U^\perp \cap \hsphere$ (which exists by assumption), and set $s = u + (1- \norm{u}^2)^{1/2} u^\perp$.
    This gives 
    \[
        \inf_{s \in \hsphere} F_\phi(s) = \inf_{u \in B} F_\phi(u).
    \]
    Take $\{ u_n\} \subset B$ a sequence minimising $F_\phi$ over $B$, i.e., $\inf_{u \in B} F_\phi(u) = \lim_{n \to \infty} F_\phi(u_n)$. By weak compactness of $B$, we can extract a subsequence $\{u_n\}$ (by slight abuse of notation) that converges weakly to $u^* \in B$. Therefore, for any $x \in \hsphere$, $\inner{u_n, x} \to \inner{u^*, x}$ as $n \to \infty$. Since $|\inner{u_n, x}| \leq 1$, $\phi( d_\hsphere(u_n, x) )$ is well-defined, and since $\phi \circ \arccos$ is continuous, 
    \[
        \phi(d_\hsphere(u_n, x))  = \phi( \arccos( \inner{u_n, x} )) \to \phi(d_\hsphere(u^*, x) ), \quad  \text{as } n \to \infty.
    \]
    Furthermore, $\phi$ being continuous on $[0,\pi]$, 
    \[
        |\phi(d_\hsphere(u_n,x))| = |\phi( \arccos( \inner{u_n, x} ) )| \leq \sup_{t \in [0, \pi]} |\phi(t)| =:M < \infty,
    \]
    thus the dominated convergence theorem yields
    \[
        \inf_{u \in B} F_\phi(u) = \lim_{n \to \infty} F_\phi(u_n) = \lim_{n \to \infty} \int_\hsphere \phi( d_\hsphere(u_n, x) ) \ud \nu(x) = \int_\hsphere \phi( d_\hsphere(u^*, x) ) \ud \nu(x) =F_\phi(u^*). 
    \]
    Now take $s^* = u^* + (1- \norm{u^*}^2)^{1/2} u^\perp$ or $s^* = u^* - (1- \norm{u^*}^2)^{1/2} u^\perp$, for any $u^\perp \in U^\perp \cap \hsphere$. We have $s^* \in \hsphere$, and $s^*$ minimises $F_\phi$ since $F_\phi(s^*) = F_\phi(u^*)$.
    Noticing that the function $t \mapsto t^p$ is continuous on $[0,\pi]$ for any $p > 0$, the $L^p$ centre of mass statement follows. 

We now turn to the proof for $L^\infty$ centres of mass.
We extend $F_\infty$ to $B$ by
\[
  F_\infty(u) = \sup_{x \in \overline{\support(\nu)}} \arccos\inner{u,x}, \qquad u \in B. 
\]
Following the same steps as above,
    \[
        \inf_{s \in \hsphere} F_\infty(s) = \inf_{u \in B} F_\infty(u).
    \]
    Take $\{ u_n\} \subset B$ a sequence minimising $F_\infty$ over $B$, i.e.,
    \[
      F_\infty(u_n) \longrightarrow \inf_{u\in B}F_\infty(u).
    \]
    By weak compactness of $B$, we can extract a subsequence $\{u_n\}$ (by slight abuse of notation) that converges weakly to $u^* \in B$.
For every fixed $x \in \overline{\support(\nu)}$, the map
\[
  u \mapsto d_\hsphere(u,x)=\arccos\inner{u,x}
\]
is weakly continuous on $B$, since $u \mapsto \inner{u,x}$ is weakly continuous and $\arccos$ is continuous on $[-1,1]$.
    The map $F_\infty$ is weakly lower semicontinuous  on $B$ (as the supremum of weakly continuous maps), hence
    \[
      F_\infty(u^*) \leq \liminf_{n\to\infty} F_\infty(u_n) = \inf_{u \in B} F_\infty(u).
    \]
    so $u^*$ achieves the infimum of $F_\infty$ over $B$.
    Now take $s^* = u^* + (1- \norm{u^*}^2)^{1/2} u^\perp$ or $s^* = u^* - (1- \norm{u^*}^2)^{1/2} u^\perp$, for any $u^\perp \in U^\perp \cap \hsphere$. We have $s^* \in \hsphere$, and $s^*$ minimises $F_\infty$ since $F_\infty(s^*) = F_\infty(u^*)$.

In both cases above, notice that if $s^*$ is unique, $\norm{u^*} = 1$ and hence $u^* \in \hsphere$.
\end{proof}

\begin{proof}[of Theorem~\ref{thm:sample}]
    The existence follows from Theorem~\ref{thm:existence} since $\dim( \vspan(x_1,\ldots, x_n) ) \leq n$.

    We now turn to the second claim. Let $U = \vspan(u, x_1,\ldots, x_n)$. For any fixed $s \in \hsphere$, we can construct an orthogonal operator $R: \H \to \H$, $R R^\adj = I$, where $R^\adj$ denotes the adjoint operator of $R$ and $I$ is the identity operator on $\H$, such that $R x_i = x_i$ for $i=1,\ldots, n$ and $Rs \in U \cap \hsphere$.
    Since $ \phi( d_\hsphere(s, x_i) ) = \phi( d_\hsphere(Rs, Rx_i) )  = \phi( d_\hsphere(Rs, x_i) )$ for $i=1,\ldots,n$, $F^{(n)}_\phi(s) = F^{(n)}_\phi(Rs)$. Noting that this holds for any $s \in \hsphere$ (but with different operators $R$ for each $s$),
    \[
        \inf_{s \in \hsphere} F^{(n)}_\phi(s) = \inf_{u \in U \cap \hsphere} F^{(n)}_\phi(u).
    \]
    For the last claim, let $\hat s \in \hsphere$ be a minimiser of $F^{(n)}_\phi$ with $\inner{\hat s, u} \neq 0$. Taking the orthogonal operator $R: \H \to \H$ that fixes the data $x_1,\ldots, x_n$ but maps $u$ to $-u$ yields a distinct minimiser $R \hat s$ of $F^{(n)}_\phi$.
    This completes the proof.
\end{proof}

\begin{proof}[of Proposition~\ref{proposition:gram_to_solution}]
  We use the notation introduced in Algorithm~\ref{algo:computation_centre_mass}.
  Since any $m \in \vspan(x_1,\ldots,x_n, u) \cap \hsphere$ can be written $m = \sum_{j=1}^n \alpha_j x_j  + \beta u$ for some coefficients $\alpha_1,\ldots, \alpha_n, \beta \in \Real$, we have
  \begin{align*}
    F_\phi^{(n)}(m) &= n^{-1} \sum_{i=1}^n \phi\left\{ d_\hsphere(m, x_i) \right\}  \\
                    &= n^{-1} \sum_{i=1}^n \phi\left\{ \arccos[ \inner{ m, x_i} ] \right\} \\ 
                    &= n^{-1} \sum_{i=1}^n \phi\left\{ \arccos\left[ \sum_{j=1}^n \alpha_j \inner{x_j, x_i} + \beta\inner{u, x_i} \right] \right\}  
                    \intertext{since $\inner{u, x_i}=0$,}
                    &= n^{-1} \sum_{i=1}^n \phi\left\{ \arccos\left[ \sum_{j=1}^n \alpha_j \inner{x_j, x_i} \right] \right\}  \\
                    &= n^{-1} \sum_{i=1}^n \phi\left\{ \arccos\left[ \sum_{j=1}^n \alpha_j y_j^\tp y_i  \right]\right\}  
                    \intertext{and since $\tilde m$ is the $\phi$-centre of mass of $y_1,\ldots, y_n$, and $v^\tp y_i = 0$ for all $i=1,\ldots, n$,}
                    &\geq  n^{-1} \sum_{i=1}^n \phi\left\{ \arccos\left[ \sum_{j=1}^n \tilde \alpha_j y_j^\tp y_i  \right] \right\} \\
                    &=  n^{-1} \sum_{i=1}^n \phi\left\{ \arccos\left[ \sum_{j=1}^n\tilde  \alpha_j \inner{x_j, x_i}  \right]   \right\}
                    \intertext{using the fact that $\inner{u, x_i}=0$,}
                    &=  n^{-1} \sum_{i=1}^n \phi\left\{ \arccos[ \inner{\hat m, x_i} ] \right\} \\ 
                    &= F_\phi^{(n)}(\hat m).
  \end{align*}
  Notice that $\hat m$ belongs to $\hsphere$ since 
  \[
    \norm{\hat m}^2 = \sum_{j,j'=1}^n \tilde \alpha_j \tilde \alpha_{j'} \inner{x_j, x_{j'}} + \tilde \beta^2  = \sum_{j,j'=1}^n \tilde \alpha_j \tilde \alpha_{j'} y_j^\tp y_{j'} + \tilde \beta^2  = \norm{\tilde m}^2  = 1 .
  \]
\end{proof}

\bibliographystyle{plainnat}
\bibliography{frechet}

\end{document}